\documentclass[11pt,reqno]{amsart}

\usepackage{amsmath,amssymb,amsthm,enumitem,xcolor,tikz,tikz-cd,subcaption,
  hyperref,mathtools,float,graphicx}
\usepackage{stackengine,scalerel}

\pgfdeclarelayer{fg}
\pgfsetlayers{main, fg}

\usepackage{cleveref}

\numberwithin{equation}{section}
\hypersetup{  pdfborder={0 0 0},
        colorlinks   = true,
        urlcolor     = blue,
        linkcolor    = blue,
        citecolor   = red}

\makeatletter
\newcommand\bigcdot{\mathpalette\bigcdot@{.5}}
\newcommand\bigcdot@[2]{\mathbin{\vcenter{\hbox{\scalebox{#2}{$\m@th#1\bullet$}}}}}

\makeatother
\newcommand{\Cc}{\mathbb{C}}
\newcommand{\N}{\mathbb{N}}
\newcommand{\Z}{\mathbb{Z}}
\newcommand{\R}{\mathbb{R}}

\newcommand{\F}{\mathbb{F}}

\newcommand{\n}{\unlhds}

\stackMath
\def\unlhds{\ThisStyle{\mathrel{%
      \stackinset{r}{.75pt+.15\LMpt}{t}{.1\LMpt}{\rule{.3pt}{1.1\LMex+.2ex}}
      {\SavedStyle\leqslant}%
}}}

\newcommand{\lqqs}{\leqslant}
\newcommand{\gqqs}{\geqslant}
\newcommand\restr[2]{{
                \left.\kern-\nulldelimiterspace 
\right                                                                               
                #1 
                \littletaller 
                \right|_{#2} 
}}

\newcommand{\littletaller}{\mathchoice{\vphantom{\big|}}{}{}{}}

\newtheorem{thm}{Theorem}[section]
\newtheorem{prop}[thm]{Proposition}
\newtheorem{lemma}[thm]{Lemma}
\newtheorem{cor}[thm]{Corollary}

\theoremstyle{definition}
\newtheorem{defn}[thm]{Definition}

\theoremstyle{remark}
\newtheorem{ex}[thm]{Example}

\title{Rank 1 phenomena for groups of positive deficiency}
\author[Button]{J.\,O.\,Button
}
\address{Selwyn College, University of Cambridge,
Cambridge CB3 9DQ, UK}


\begin{document}
\vspace*{-1cm}

\begin{abstract}
We examine second bounded cohomology and mod $p$ homology in finite index
  subgroups of 1-relator groups and groups with a presentation of
  positive deficiency. We use this to
  determine exactly which subgroups of
  1-relator groups are boundedly generated.
  \end{abstract}

        \maketitle

    \setcounter{tocdepth}{1}
\tableofcontents


\section{Introduction}

Let $\Lambda$ be a semisimple Lie group, equipped with Haar measure.
(In this paper all Lie groups will be in characteristic zero.) It is
then the case that any lattice $G$ of $\Lambda$, 
whether uniform or non-uniform, will be finitely presented.
  The {\bf (real) rank} $r$ of $\Lambda$ is the maximal $r$ for which
  there exists an embedding $\R^r$ into $\Lambda$. We also call this
  the (real) rank of the lattice $G$ when the ambient Lie group is clear. 
  
It has long been known that in semisimple Lie groups, rank 1 lattices and
lattices of higher rank behave very differently, even just when considering
their group theoretic properties. Here we implicitly assume throughout that
all lattices are irreducible (namely no finite index subgroup
  splits as a direct product of infinite groups),
otherwise a direct product
$G\times \ldots \times G$ of a rank 1 lattice $G$
would count as higher rank but would generally
have similar properties to $G$. 
For instance, if we were to take the rank 1
group $SL_2(\Z)$ and the higher rank group
$SL_n(\Z)$ where $n\gqqs 3$, we could
distinguish them in many ways. To give but a couple of examples
here, $SL_n(\Z)$ has
distorted infinite cyclic subgroups unlike $SL_2(\Z)$, but $SL_2(\Z)$ is
virtually free whereas $SL_n(\Z)$ has property (T) and all of its non-trivial
normal subgroups have finite index.

Let us pick a group theoretic property $P$.
This means that if we have two abstract groups $G$ and $H$ which
are isomorphic then $G$ has $P$ if and only if $H$ has $P$. We can then ask
whether the presence of $P$ in a finitely presented group is consistent
with its presence in
rank 1 lattices and its absence is consistent with higher rank
lattices (or vice versa but then we could just replace $P$ with its
negation), as in the examples above. Having picked such a $P$, we can then
take our favourite family $\mathcal{F}$ of finitely presented groups
(which might have nothing to do with lattices a priori) and ask which groups
$G$ in $\mathcal{F}$ have this property $P$, whereupon we can say that here
$G$ exhibits rank 1 behaviour, and which groups in $\mathcal{F}$
do not, whereupon we can describe this as higher rank behaviour.
Ideally our hope is that either every group in $\mathcal{F}$ has $P$ or that
no group does, in which case we can say that $\mathcal{F}$ itself
exhibits rank 1/higher rank behaviour.

This was the philosophy adopted in \cite{flm} where the family $\mathcal F$
being considered was the mapping class groups $Mod(\Sigma_g)$
of closed orientable connected surfaces $\Sigma_g$ of genus
$g$ where $g\gqqs 1$. The authors stated there that
``A well-known dichotomy among the lattices in simple Lie groups is between
lattices in rank-1 groups and higher rank lattices, that is, those lattices
in simple Lie groups of $\R$-rank at least 2. It is somewhat mysterious whether
$Mod(\Sigma_g)$ is similar to the former or the latter.''
However the paper then went on to find rank 1 properties
satisfied by all groups in $\mathcal F$. Those properties were referred to
as rank 1 phenomena and we will adopt that term here. They first showed
that for each group in $\mathcal F$, every element of
infinite order is undistorted which was mentioned above as a rank 1
phenomenon, although for us this will be too restrictive. 
The other rank 1 phenomenon examined in \cite{flm} involves pro-$p$
completions, so we now give some guiding definitions and refer to the book
\cite{gz2} for more information.

\begin{defn} \label{prop}
  A {\bf profinite} group is the inverse limit of an inverse system of groups
  which are finite and a {\bf pro-$p$} group is the same but where the
  groups are finite $p$-groups.

  We say a profinite or pro-$p$ group $\Gamma$ is {\bf finitely generated}
  if there
  exists a finite set $S$ which topologically generates $\Gamma$
  (meaning that the closure of the abstract group generated by $S$ is
  $\Gamma$). 

    If $G$ is any abstract group then let $d(G)$ be the minimum number
  of elements required to generate $G$, so $d(G)$ is finite if and only
  if $G$ is finitely generated.
 We also use $d(G)$ if $G$ is a profinite or more specifically a pro-$p$
 group, where now this means the minimum number of elements for
 topological generation.
  
Let $\Gamma$ be a finitely generated pro-$p$ group.
     We say that $\Gamma$ has {\bf finite Pr\"ufer rank} if there is a finite
     upper bound for $d(H)$ as $H$ varies over all closed subgroups of
     $\Gamma$ (equivalently over all open subgroups), treating $H$ as a
     pro-$p$ group. We also say in this
case that $\Gamma$ is a {\bf $p$-adic analytic group}. We do not
define that concept here because in practice we will use the Pr\"ufer rank 
definition, but they are shown to be equivalent in \cite[Interlude A]{ddms}.

The {\bf profinite} or {\bf pro-$p$ completion} $\hat{G}$ or $G_{\hat{p}}$
of an abstract finitely generated group $G$ is the inverse limit of the
inverse system of groups consisting of the finite or finite-$p$ quotients
of $G$.
\end{defn}

The idea here is that amongst all pro-$p$ groups, $p$-adic analytic
groups are considered to be ``small''.
The particular rank 1 phenomenon shown in \cite[Theorem 1.4]{flm}
to hold for any $Mod(\Sigma_g)$ with
$g\gqqs 1$ was the following.
\begin{defn} \label{rk1p}
Given a finitely presented group $G$, we say that $G$ has
{\bf rank 1 pro-$p$ behaviour} if for every prime $p$, the group
$G$ has a finite index subgroup $H$
where the pro-$p$ completion
$H_{\hat{p}}$ is not $p$-adic analytic. Otherwise we say that
$G$ has {\bf higher rank pro-$p$ behaviour}.
\end{defn}

However we can work with a weaker version of this property.
\begin{defn} \label{rk1pr}
We say that $G$ has
{\bf rank 1 pro-prime behaviour} if $G$ has a finite index subgroup $H$
such that there exists a prime $p$ where the pro-$p$ completion
$H_{\hat{p}}$ is not $p$-adic analytic. Otherwise we say that
$G$ has {\bf higher rank pro-prime behaviour}.
\end{defn}

Now another property shared by mapping class groups and rank 1 lattices
but which fails for higher rank lattices is the following
(see the book \cite{frig} for the
definition and quoted results). 
\begin{defn} \label{rk1bc}
Given a finitely presented group $G$, we say that $G$ has
{\bf rank 1 bounded cohomology behaviour} if $G$ has a finite index subgroup 
$H$ where $H$ has infinite dimensional second bounded cohomology.
If all finite index subgroups of $G$ have finite dimensional second bounded
cohomology then we say that $G$ has {\bf higher rank bounded cohomology
behaviour}.
\end{defn}

We now introduce the classes of finitely presented groups $\mathcal F$
that we will consider in this paper.
\begin{defn}
We say that a finitely presented group $G$ has {\bf positive deficiency}
if there exists a presentation $\langle X\,|\,R\rangle$ for $G$ with
$|X|>|R|$, namely there are more generators then relators.

We say that $G$ has {\bf deficiency at least two} if it is defined by
some presentation   $\langle X\,|\,R\rangle$ with
$|X|-|R|\gqqs 2$.

We say that $G$ is a {\bf 1-relator} group if it is defined by some
presentation $\langle X\,|\,R\rangle$ with
$|R|=1$.
\end{defn}

Note that if $G$ is 1-relator then $G$ has positive deficiency unless it is
given by a presentation of the form $\langle x\,|\,r\rangle$ for $x$ a
single generator and $r$ a single relator, in which case it is cyclic.
 
It is the case in practice that not every member of a given family of groups 
$\mathcal F$ might satisfy a particular rank 1 phenomenon $P$ but that the
only exceptions are groups that are too ``small'' to assign rank 1 or higher
rank behaviour to them. For instance in the case defined above when
$\mathcal F$ is all 1-relator groups, it will not worry us here whether
we regard cyclic groups as rank 1 or higher rank. Similarly the   
the family $\mathcal F$ of mapping class groups considered in \cite{flm}
did not contain the mapping class group of the genus 0 surface.
Our viewpoint here will be that if all groups in our favourite
family $\mathcal F$ have this property $P$ with the exception only of those
that are clearly ``small'' (for instance being amenable or being soluble)
then we will still say that the family $\mathcal{F}$
satisfies the rank 1 phenomenon $P$.

In this paper our results are as follows. The first theorem
is implied by existing results in the literature.
\begin{thm} \label{def2} (see Proposition
  \ref{df2bd2} and Corollary \ref{rk1pd2})
If $G$ is a group of deficiency at least two then $G$ has both rank 1 bounded
cohomology behaviour and rank 1 pro-$p$ behaviour.
\end{thm}

In fact here we do not need to drop to finite index subgroups. But
the situation is more involved for groups of positive deficiency where
``small'' groups will certainly make an appearance. We have the well
known {\bf Baumslag -- Solitar groups} defined by the presentation 
$BS(m,n)=\langle x,y\,|\,yx^my^{-1}=x^n\rangle$ where without loss of
generality we can take the integers $m,n$ to satisfy $m\gqqs 1$ and
$|n|\gqqs m$. If $m=1$ then $BS(1,n)$ is soluble and it
is known (see for instance \cite{wls}) that these and $\Z$
are the only (virtually) soluble groups of positive deficiency.

However if $G$ is a group of positive deficiency then it surjects to $\Z$
and thus splits as an HNN extension with base (vertex group)
$H$ and associated (edge) subgroups $A,B$ of $H$, using an isomorphism
$\phi:A\rightarrow B$. We then say that the HNN extension
is {\bf ascending} if $H=A$ and {\bf strictly ascending} if $H=A>B$.

\begin{thm} \label{def1int} (see Theorem \ref{psdfthm})
  If $G$ is a group defined by a presentation of positive deficiency then\\ 
$\bullet$ either $G$ is isomorphic to $BS(1,n)$ or to $\Z$\\ 
$\bullet$ or $G$ satisfies at least one of rank 1 bounded
cohomology behaviour or rank 1 pro-prime behaviour\\
$\bullet$ or $G$ can be expressed as a strictly ascending HNN extension with
base the finitely generated group $H$ where $H$ is not $FP_2$.
\end{thm}
The third case is conjectured in \cite{bi} not to exist.

When our family consists of all 1-relator groups, we can avoid 
the case above which is not known to exist and hence obtain an unconditional
result. In fact because of recent advances on the coherence of 1-relator
groups, we can deal with subgroups too.
\begin{thm} \label{rel1int} (see Theorem \ref{rel1sb})
  If $H$ is any finitely generated subgroup of a group $G$ defined by a
  1-relator presentation then\\
$\bullet$  either $H$ is cyclic or the infinite dihedral group $C_2*C_2$\\
$\bullet$ or $H$ is isomorphic to $BS(1,n)$ for $n\in\Z\setminus\{0\}$\\
$\bullet$ or $H$ satisfies at least one of rank 1 bounded
cohomology behaviour or rank 1 pro-prime behaviour.
\end{thm}

Our applications are to bounded generation.
\begin{defn} 
  A group $G$ is boundedly generated if there exists a finite family
  of cyclic subgroups $C_1,\ldots ,C_M$ of $G$ such that
  $G=C_1\ldots C_M$.
\end{defn}
See Definition \ref{bgdf} for equivalent definitions.

We refer to Section 6 for examples and references. Here we note that
virtually polycyclic groups are boundedly generated. Moving
away from such ``small'' groups,
an important example is that $SL_2(\Z)$ is not boundedly generated
but $SL_3(\Z)$ is, using the elementary matrices. This
suggests that bounded generation looks likes a higher rank phenomenon,
though this does not stand up in general because it was proven in \cite{invt}
that if $\Lambda$ is a linear Lie group then no uniform lattice in $\Lambda$
can be boundedly generated, regardless of being rank 1 or higher rank.
Nevertheless, it is the case that both the phenomena of having rank 1
pro-prime behaviour and of having rank 1 bounded cohomology behaviour
imply the absence of bounded generation (this is indicated in Section
6). We are then able to use the above theorems to obtain these results:
\begin{cor} \label{def1bgint} (see Corollary \ref{psdfbg})
  Suppose that $G$ is a group defined by a presentation of positive deficiency.
  Suppose further that $G$ is not a strictly
  ascending HNN extension with
  base a finitely generated group $H$ where $H$ is not $FP_2$.
    Then $G$ is boundedly generated exactly when
$G$ is isomorphic to $BS(1,n)$ or to $\Z$.
\end{cor}

\begin{cor} \label{rel1bgint} (see Corollary \ref{1ls})
  If $H$ is any subgroup of a group
  $G$ defined by a
  1-relator presentation then $H$ is boundedly generated exactly
  when it is cyclic or $C_2*C_2$ or is isomorphic to $BS(1,n)$ for
  some $n\neq 0$.
\end{cor}
  
The outline of
the paper is that we look at when a group has infinite dimensional
second bounded cohomology in Section 2 and relate this to having
an infinite dimensional space of homogeneous quasimorphisms. There
are various results available on second bounded cohomology when a
group splits. Here if $G$ is a group of positive deficiency and thus
splits as an HNN extension, we can apply these results when this
HNN extension is non-ascending to obtain infinite dimensional second bounded
cohomology in Corollary \ref{toz}.

In Section 3 we examine our other rank 1 phenomenon concerning pro-$p$
completions of finite index subgroups of $G$. Again deficiency at least two
is easily taken care of and for positive deficiency groups, the
phenomenon of rank 1 pro-prime behaviour
takes on a particularly tractable form in Definition \ref{vsa} which we
name virtually sufficient abelianisation.

When $G$ is an ascending HNN extension of the finitely generated group $H$,
we divide into two cases: not strictly ascending which implies that
$G=H\rtimes\Z$ and when $G$ is strictly
ascending. The first case is dealt with in
Section 4. We have from known results that $H$ must be free and here
a range of different arguments are available to show that $G$ has both
of our rank 1 phenomena. The heart of the paper is in Section 5 on
strictly ascending HNN extensions. The first problem here is that we
do not always know $H$ will be free, hence the potential
exceptions above. But even if $H=F_r$ is free
of finite rank $r$, the family $\mathcal S$ of
strictly ascending HNN extensions of $F_r$
can be harder to deal with than the free by cyclic groups $F_r\rtimes\Z$.
Although some general results are known for groups $G$ in
$\mathcal S$, for instance
\cite{bs} and \cite{mut}, a particularly intractable case is when 
$G$ contains a Baumslag - Solitar group $BS(1,n)$ for $|n|\gqqs 2$ whereupon
the geometry of $G$ will not be well behaved.

Therefore we proceed by showing in Theorem \ref{main} that for $r\gqqs 2$,
any ascending HNN extension of $F_r$ containing $BS(1,n)$
possesses the rank 1 phenomenon of virtually sufficient abelianisation.
In doing this, we adapt an argument in an earlier version of
\cite{lgv4} to show in Theorem
\ref{melg} that if an injective endomorphism of $F_r$ sends some
non-identity element $w\in F_r$ to a conjugate of $w^r$ then $w$ is a
power of a primitive element.

So we conclude that one or other of our rank 1 phenomena holds for every
non-small group of positive deficiency (away from the exceptional
case that is not known to occur) and in Section 6 we apply
this to the property of bounded generation.
Our final section is on 1-relator groups. Here we can show that
this exceptional case
cannot occur, so one or other of our two rank 1 phenomena does hold for
every non-small 1-relator group. In fact we actually show that this
remains true for any finitely presented subgroup of a 1-relator group.
As coherence is now know to hold for 1-relator groups, this deals with
bounded generation for any subgroup of a 1-relator group.

We finish by investigating
what happens in the 1-relator case
if we strengthen the second bounded cohomology phenomenon to
being virtually acylindrically hyperbolic. Now we would expect that
any non-cyclic 1-relator group which is neither virtually acylindrically
hyperbolic nor has virtually sufficient abelianisation
would be isomorphic to $BS(m,n)$ for $m,n$ coprime, or a finite index
subgroup thereof.
But we find all such 1-relator groups $G$ in Theorem \ref{last}
and show that there are other examples.

The author would like to thank F.\,Fournier - Facio for help with references
and the anonymous referee for many useful suggestions concerning the original
submission.

\section{Bounded cohomology}

Here we are interested in the second bounded cohomology of a finitely
presented group $G$ and its connection
to quasimorphisms. We give brief details and
definitions because this is an area that is well covered in the
literature: see \cite{frig} and references therein for more detail.

Given any group $G$, let $H^2_b(G,\R)$ be the second bounded cohomology of
$G$, which is an $\R$-vector space. We have the comparison map $c$ which
is a linear map from $H^2_b(G,\R)$ to the usual cohomology $H^2(G,\R)$.

\begin{defn} \label{quasi}
Given a group $G$, a {\bf quasimorphism} on $G$
is a function $q:G\rightarrow\R$ such that there exists
$D\gqqs 0$ where
\[|q(gh)-q(g)-q(h)|\lqqs D\mbox{ for all }g,h\in G.\]
The least such value $D$ is the {\bf defect} of the quasimorphism.

A {\bf homogeneous quasimorphism} on $G$ is a quasimorphism
such that for all $g\in G$ and $n\in\Z$ we have $q(g^n)=nq(g)$.
The set of homogeneous quasimorphisms (of arbitrary defect) forms
a vector space under addition and scalar multiplication, denoted
$Q_h(G)$.
\end{defn}

If we let $H^1(G,\R)$ be the vector space of genuine homomorphisms
then the kernel $EH^2_b(G,\R)$ of the comparison map $c$ (here $E$
stands for exact) can
be identified with the quotient vector space $Q_h(G)/H^1(G,\R)$.
Thus if $G$ is finitely presented then the vector space $Q_h(G)$ has
infinite dimension if and only if $H^2_b(G,\R)$ does,
because $G$ being finitely presented means that the usual cohomology
$H^2(G,\R)$ has finite dimension.

We now state some standard results on second bounded cohomology.
First we give some groups where this cohomology has infinite
dimension.

\begin{prop} \label{inf2}
The following groups have an infinite dimensional space of
homogeneous quasimorphisms and hence infinite dimensional
second bounded cohomology.\\
(i) Non-abelian free groups (see \cite{brks}).\\
(ii) Any non-elementary word hyperbolic group (see \cite{epfuj}).\\
(iii) Any acylindrically hyperbolic group (see \cite[Theorem 8.3]{os13}).
\end{prop}

Note that (ii) includes all uniform rank 1 lattices and
(iii) includes the mapping class groups $Mod(\Sigma_g)$ for $g\gqqs 1$.
It also includes any non-uniform
rank 1 lattice, which will be relatively hyperbolic with respect to
the cusp subgroups and hence acylindrically hyperbolic.

On the other hand there are groups with small second bounded cohomology.

\begin{prop} \label{no2}
The following groups have finite dimensional second bounded cohomology.\\
(i) All amenable groups (see \cite{jnbe} which shows that the
cohomology is trivial).\\
(ii) Irreducible
higher rank lattices (see \cite[Corollary 1.3]{bgmn99} for the
uniform case and \cite[Theorem 21]{bgmn02} for the non-uniform case)\\
(iii) Groups of orientation preserving piecewise linear
homeomorphisms of the closed unit interval (see \cite{clgpac}
for finitely presented groups and \cite{ffl} in general).
\end{prop}

We also note here some stability properties occurring when a group has
infinite dimensional second bounded cohomology.
\begin{prop} \label{perm}
(i) If $\theta:G_1\rightarrow G_2$ is a surjective  homomorphism then there
is a well defined linear map $H^2_b(\theta):H^2_b(G_2,\R)\rightarrow H^2_b
(G_1,\R)$ which is injective (\cite{frig} Theorem 2.17]).\\
In particular if $G$ is a group which surjects to any of the groups in
Proposition \ref{inf2} then $G$ has infinite dimensional second bounded
cohomology.\\
(ii) If $H$ is a finite index subgroup of $G$ then $H^2_b(G,\R)$
embeds in $H^2_b(H,\R)$ by \cite[Example 12.4.3]{mnd}. Thus if $G$
has infinite dimensional second bounded cohomology then so do all of
its finite index subgroups.
\end{prop}

However it might be noted in Definition \ref{rk1bc} of our rank 1 phenomenon
that we say our group $G$ has a finite index subgroup $H$ having infinite
dimensional second bounded cohomology and not $G$ itself. This is
because it is unknown whether $H$ having infinite dimensional second
bounded cohomology implies that $G$ does. It is certainly true that
dim $H^2_b(G,\R)$ can be less than dim $H^2_b(H,\R)$. For instance,
\cite[Remark 2.4]{fran} gives an example of a finitely presented group $G$
with vanishing second bounded cohomology which possesses an index two
subgroup $N$ with non-vanishing (but finite dimensional) second bounded
cohomology. However it is shown in \cite[Proposition 1.6]{kkmm22} that
if $N\n_f G$ and we have a homogeneous quasimorphism $q$ on $N$ which
is invariant, meaning that $q(gng^{-1})=q(n)$ for all $g\in G,n\in N$,
then $q$ lifts to a homogeneous quasimorphism on $G$. Moreover the paper
\cite{ffw23} 
mentions many examples of finitely presented groups $N$ which have
an infinite dimensional space of homogeneous quasimorphisms $q$ that are
invariant under $Aut(N)$; that is $q(\alpha(n))=q(n)$ for any $n\in N$ and
$\alpha\in Aut(N)$. As conjugation by any element of $g$ will induce
an automorphism on $N$, this infinite dimensional space will lift to $G$.
Therefore we will not worry too much about this distinction.

Results as in Proposition \ref{inf2}
are often obtained by analysing a suitable action of our group $G$
on some hyperbolic space. This suggests that further progress can be
made by specialising to actions on simplicial trees. This is indeed
the case as can be seen in \cite{sstocan}
(with a correction for one of their cases in \cite{hffmstrs}). There are
two earlier results that can be found in \cite{kj2} and also (without
proof) \cite{grig}, one for amalgamated
free products and one for HNN extensions. We will just need the latter.
\begin{thm} (\cite[Theorem 1.2]{kj2}, see also \cite{grig})
\label{kj2hnn}
Suppose that $G$ is any group that splits as an HNN extension
$A*_\phi$
  (which is equivalent to $G$ having a surjection to $\Z$) such that the
  the edge groups $C, \phi(C)$ are both
  properly contained in the vertex group $A$. Then $G$
  has infinite dimensional second bounded cohomology.
\end{thm}

  \begin{cor} \label{toz}
    If $G$ is any finitely presented group with a surjection to $\Z$ but
    with finite dimensional second bounded cohomology
    then $G$ is an ascending HNN extension
    of a finitely generated group.
  \end{cor}
  \begin{proof} Any HNN extension $G$ which is finitely presented, or even
    is $FP_2$ over the field $\F_2$,
    can be
    ``folded'', which involves changing the HNN extension but not $G$ itself,
    to an HNN extension of some group $A$ with edge subgroups
    $C$ and $\phi(C)$ of $A$
    where $A$ and $C$ (thus $\phi(C)$ too) are finitely generated
    (see any of \cite{bistr} or \cite{bsh} or \cite[VI Theorem 4.5]{dnwdybk}
    for instance). If neither $C$
    nor $\phi(C)$ are equal to $A$ then $G$ has infinite dimensional
    second bounded cohomology by Theorem \ref{kj2hnn}. Otherwise we are now in
    the conclusion of the corollary, swapping $C$ with $\phi(C)$ and
    replacing $\phi$ with $\phi^{-1}$ if necessary.
   \end{proof}       

Let us start our applications by considering groups of deficiency at
least two.
\begin{prop} \label{df2bd2}
Suppose that the group $G$ has a finite presentation
of deficiency at least two. Then $G$ has rank 1 bounded cohomology
behaviour as in Definition \ref{rk1bc} and even has infinite dimensional
seconded bounded cohomology itself.
\end{prop}
\begin{proof} If we just want the first part then we can quote
  the famous result \cite{bp} of Baumslag and Pride, saying that
  $G$ is large; namely $G$ has a finite
  index subgroup possessing a surjective homomorphism to $F_2$. We are then
  done by Proposition \ref{inf2} (i) and Proposition \ref{perm} (i).

  Alternatively \cite[Corollary 1.3]{os1501arxd} shows that $G$ is
  acylindrically hyperbolic and so we are done by \ref{inf2} (iii).
  Note that being acylindrically hyperbolic is another property not
  known to pass to finite index supergroups (see
  \cite{motr}).
 \end{proof} 

 We now move to groups of positive deficiency. The next definition
 establishes the nomenclature we will use in this paper, as there is
 some variation in the literature.

\begin{defn} \label{gbsdf}
Given $n\gqqs 1$,  a {\bf generalised
    Baumslag - Solitar group of rank $n$} is the fundamental group of a finite
  graph of groups where each edge and vertex group is a copy of $\Z^n$.
If $n=1$ we call such a group $G$
a  {\bf generalised Baumslag - Solitar group}. When $n=1$, the natural
presentation for $G$ obtained from the graph of groups description will
have deficiency 1.
\end{defn}

  \begin{thm} \label{gbs}
    Let $G$ be a generalised Baumslag - Solitar group of arbitrary rank
    $n\gqqs 1$ with underlying graph
    $\Gamma$. Assume without loss
    of generality that the finite graph of groups is reduced; that is the
    inclusion of any edge group into any adjacent vertex group is proper
    unless this edge is a self loop.\\
    (i) $G$ is not acylindrically hyperbolic and nor is any
    finite index subgroup of $G$.\\
    (ii) $G$ does have infinite dimensional second bounded cohomology
    unless $\Gamma$ is a single point, or $\Gamma$ is a single
    edge with index 2 inclusions on both sides, or $\Gamma$ is a single
    loop with (at least) one inclusion of the edge group being equal to the
    vertex group. In these exceptional cases, $G$ is soluble so
    all its finite index
    subgroups have trivial second bounded cohomology.
 \end{thm}   
 \begin{proof}
   The case of finite index subgroups of $G$ in (i)
   is easily taken care of since
   a finite index subgroup $H$ of $G$ will act on the Bass - Serre tree with
   vertex/edge stabilisers that are finite index subgroups of the vertex/edge
   stabilisers of $G$, so $H$ is still a generalised Baumslag - Solitar group
   of the same rank $n$.
   
   The without loss of generality comment in the statement holds because we can
   perform an elementary collapse, that is we can contract any edge which is
   not a self loop if its edge group is equal to a vertex group on one side.
   We then relabel the edge inclusions on the other side, so that we still
   have a generalised Baumslag - Solitar group of the same rank, and this
   does not change the fundamental group.

   To show $G$ is not acylindrically hyperbolic, note that any vertex group
   $V$ (which is a copy of $\Z^n$) is commensurable with $gVg^{-1}$ for an
   arbitrary $g\in G$, because the inclusions are all finite index, so the
   associated Bass - Serre tree has finite valency but $V$ and $gVg^{-1}$
   are both vertex stabilisers. In the language of \cite{motr} $V$ is
   an $s$-normal subgroup of $G$, meaning that $V\cap gVg^{-1}$ is infinite
   for all $g\in G$. However \cite[Theorem 3.7 (c)]{motr} tells us that any
   $s$-normal subgroup of an acylindrically hyperbolic group has to be
   acylindrically hyperbolic itself.

   For (ii), note that whenever we form an HNN extension or a non-trivial
   amalgamation $A*_CB$ (non-trivial here means that the edge group $C$
   is a proper subgroup of both $A$ and $B$)
   then the resulting group contains the vertex group(s)
   with infinite index. We now consider the generalised Baumslag - Solitar
   group $G$ case by case.

   If the defining graph $\Gamma$ is not a tree then to form our fundamental
   group $G$ we take a maximal tree in $\Gamma$ and contract each edge one
   by one, forming amalgamations each time, until we are left with a single
   point with self loops whereupon we create an HNN extension with each
   loop. Thus unless $\Gamma$ has only one vertex, when we take the first
   self loop in order to form the first HNN extension we have that at both
   ends of this loop, the edge group is contained in the vertex group with
   infinite index. This property will then continue to hold as we move through
   the other self loops. Thus we are done by Theorem \ref{kj2hnn} unless
   we had one vertex and we are also done if $\Gamma$ consisted of one
   vertex with more than one self loop, because then $G$ surjects to $F_2$.

   If $\Gamma$ is just a single self loop then we are again done by
   Theorem \ref{kj2hnn} unless one of the inclusions is equal to the
   vertex group. In other words $G$ is just an ascending HNN extension
   of the soluble group $\Z^n$, so will itself be soluble and thus amenable.
   Therefore it will have trivial second bounded cohomology by
   Proposition \ref{no2} (i), as will all its subgroups.

   Otherwise $\Gamma$ is a tree. In this case we can intersect all of
   the finitely many vertex and edge subgroups of $\Gamma$ 
   to obtain a subgroup $N$ contained in each vertex group $V_i$ and
   each edge group $E_{ij}$ with finite index.
   As these are all abelian, $N$ is normal in $V_i$
   and in $E_{ij}$. Thus $N$ is
   normal in $G$ itself because the vertex groups generate $G$,
   since here $\Gamma$ is a tree and so no HNN extensions are formed.
   Thus $G/N$ can be seen as the fundamental group of a finite graph of
   groups with
   vertex groups $V_i/N$ and edge groups $E_{ij}/N$. But all vertex and
   edge groups in $G/N$ are finite, hence $G/N$ is
   virtually free and so is word hyperbolic. Therefore $G/N$, thus $G$,
   has infinite dimensional second bounded
   cohomology unless $G/N$ is virtually
cyclic and so $\Gamma$ is a point (whereupon
$G=\Z^n$), or a single edge with index 2 inclusions on either side,
in which case again $G=\Z^n*_C\Z^n$ is soluble.
   \end{proof}

   In the case $n=1$, our only exceptions for $G$ will be $\Z$
   and the soluble Baumslag - Solitar groups $BS(1,n)$.
   We have already mentioned that
   these groups are precisely the virtually soluble groups of positive
   deficiency. One might also note,
   say by combining Theorem \ref{gbs} (i)
   and the result quoted in Proposition \ref{df2bd2} on acylindrical
   hyperbolicity, that no generalised
   Baumslag - Solitar group of any rank can be defined by a presentation of
   deficiency at least two.

\section{Profinite and pro-p techniques}

   The mapping
   class group $\mathcal{M}$ (say of a closed orientable surface with genus
   $g\gqqs 1$) has infinite dimensional
   second bounded cohomology. This was first shown in \cite{bstkj})
   although it can also be seen as a consequence of Proposition \ref{inf2}
   due to its action on the
   curve complex. This can be viewed as a rank 1 phenomenon, in that it
   is shared with rank 1 lattices of semisimple Lie groups but not
   with higher rank lattices. Another rank 1 phenomenon for $\mathcal{M}$
   is shown in \cite[Theorem 1.4]{flm}. This is that for every prime $p$,
   $\mathcal{M}$ has a finite
   index subgroup $H$ whose pro-$p$ completion $H_{\hat{p}}$ is not a
   $p$-adic analytic group, as in Definition \ref{rk1p}.
   The background to this is
   explained in \cite[Chapter 4]{lssg}
   which is a wide ranging summary of results in this area. In order to
   apply this material, we now relate properties of a finitely presented
   group $G$ with its pro-$p$ completions.

   \begin{defn} \label{pad}
     A {\bf finite presentation} $\langle \mathcal{X}\,|\,\mathcal{R}\rangle$
     for a pro-$p$ group $\Gamma$ consists of a finite set $\mathcal{X}$ and
     a finite collection $\mathcal{R}$ of elements of the free pro-$p$ group
     $\mathcal{F}_\mathcal{X}$ such that $\Gamma$ is the quotient of
     $\mathcal{F}_\mathcal{X}$ by the topological closure of the abstract
     normal closure of the elements of $\mathcal{R}$.
     
       Suppose that $\Gamma$ is a finitely presented pro-$p$
       group, given by the finite pro-$p$ presentation
       $\Gamma=\langle \mathcal{X}\,|\,\mathcal{R}\rangle$. We say that
       $\langle \mathcal{X}\,|\,\mathcal{R}\rangle$
       is a {\bf minimal presentation} for $\Gamma$ if $|\mathcal{X}|$ is equal
       to the minimal number of (topological) generators for $\Gamma$,
       that is $|\mathcal{X}|=d(\Gamma)$.

       We say that $\Gamma$ satisfies the {\bf Golod -- Shafarevich
         inequality} if we have $|\mathcal{R}|\geq |\mathcal{X}|^2/4$ for every
       minimal presentation of $\Gamma$.
     \end{defn}

     It is the case that if we are given any finite presentation
     $\langle \mathcal{X}\,|\,\mathcal{R}\rangle$ for $\Gamma$ then there
       exists a minimal presentation for $\Gamma$ with the same pro-$p$
       deficiency; that is we have $\Gamma=
     \langle \mathcal{Y}\,|\,\mathcal{S}\rangle$
     where $|\mathcal{Y}|=d(\Gamma)$ and
     $|\mathcal{Y}|-|\mathcal{S}|=|\mathcal{X}|-|\mathcal{R}|$.
          (See for instance \cite[Window 5 Proposition 2]{lssg}.)   

          We now invoke a powerful theorem of Zelmanov.
     \begin{thm} (\cite{zl}) \label{zel}
       Suppose that $\Gamma$ is a finitely presented pro-$p$ group
       with $d(\Gamma)>1$. If $\Gamma$ fails to satisfy the Golod --
       Shafarevich inequality then $\Gamma$ contains a non-abelian free
       pro-$p$ subgroup $H$ as a closed subgroup.
            \end{thm}

 An immediate consequence of this theorem and Definition \ref{prop}           
 is that if $\Gamma$ fails to satisfy the Golod -- Shafarevich inequality
 then $\Gamma$ cannot be $p$-adic analytic because this closed free
pro-$p$ subgroup $H$
will itself have (finite index) closed free pro-$p$ subgroups $S$
where $d(S)$ is arbitrarily big, so $\Gamma$ does not have finite
Pr\"ufer rank.

We need to relate these concepts to abstract groups.
Suppose we have a finite presentation
     $\langle X\,|\,R\rangle$ for an abstract group $G$, so here the
     deficiency of this presentation is $|X|-|R|$. For any prime $p$,
     this can be regarded as a pro-$p$ presentation for the pro-$p$
     completion $G_{\hat{p}}$ (see \cite[Proposition 6.5.12]{gz2}).
     Let us set $d_p:=d(G_{\hat{p}})$ which is the
     the minimum number of generators
     for $G_{\hat{p}}$. Then, as mentioned above, there exists a pro-$p$
     presentation $\langle \mathcal{Y}\,|\,\mathcal{S}\rangle$ for $G_{\hat{p}}$
     where $|\mathcal{Y}|=d_p$ and the
     deficiency $|\mathcal{Y}|-|\mathcal{S}|$ is the same as our original
     deficiency $|X|-|R|$.
     
     Now for a finitely generated group $G$ we have that $d_p$ is equal to the
     rank of $H_1(G,\F_p)$ (say by \cite[Proposition 4.1.3]{gz2}). But we can
     write the abelianisation $G/G'$ of $G$ as
     $\Z^r\oplus C_{c_1} \oplus \ldots C_{c_s}$ for
     $1< c_s|c_{s-1}|\ldots |c_1$ and where $r$ is the {\bf torsion free rank}
of the abelianisation of $G$, which is at least the deficiency of any
presentation for $G$. We then have that $d(G/G')=r+s$ and
     the rank of $H_1(G,\F_p)$ is equal to $r+i$, where $p|c_i$ but 
$p\nmid c_{i+1}$.
     In particular, as there must exist some prime $p$ dividing $c_s$, we have
     that $d(G/G')=\mbox{rank }H_1(G,\F_p)=d_p$ for those values of $p$.

     This gives us a useful result for groups of positive deficiency. 
     \begin{cor} \label{ineq}
       Let $G$ be a finitely presented group with presentation
       $\langle X \,|\,R\rangle$. If $d(G/G')>1$ and
       $d(G/G')-d(G/G')^2/4$ is less than the deficiency $|X|-|R|$ of
       this presentation then there exists a prime $p$ such that the pro-$p$ 
completion $G_{\hat{p}}$ is not $p$-adic analytic.
     \end{cor}
     \begin{proof}
       We said above that there will be a prime $p$ such that 
$1<d(G/G')=d_p=\mbox{rank }H_1(G,\F_p)$, so take such a $p$.
       We then have a pro-$p$ presentation
       $\langle \mathcal{Y}\,|\,\mathcal{S}\rangle$ for $G_{\hat{p}}$
       with $|\mathcal{Y}|=d_p>1$ and $|\mathcal{Y}|-|\mathcal{S}|=|X|-|R|$
     as guaranteed earlier.
     Now the Golod - Shafarevich inequality then becomes
     $d_p^2/4\lqqs |\mathcal{S}|=d_p+|R|-|X|$ which is equivalent here to
     $d(G/G')-d(G/G')^2/4\gqqs |X|-|R|$. So if this fails, as is our hypothesis
     in the statement of the corollary, we have that Theorem \ref{zel}
     applies and so $G_{\hat{p}}$ is not $p$-adic analytic for those primes 
$p$ such that $d_p=d(G/G')$.
   \end{proof}

We can now see that groups of deficiency at least two always have our
rank 1 bounded cohomology behaviour as in Definition \ref{rk1p}.

\begin{cor} \label{rk1pd2}
  For any prime $p$, the pro-$p$ completion of
  any group $G$ defined by a presentation of deficiency at
  least two is not $p$-adic analytic.
\end{cor}
\begin{proof}
  By abelianising we have that $G/G'$ has torsion free rank at least 2,
  so that $d_p\gqqs 2$ for all primes $p$. We also have a presentation
  $\langle X\,|\,R\rangle$ for $G$ with $|X|-|R|\gqqs 2$. So our
  inequality $d_p-d_p^2/4\gqqs |X|-|R|$ always fails and we
  are done by Corollary \ref{ineq}.
\end{proof}

In relation to the
question of whether groups of positive deficiency can actually be lattices,
we first note that 
\cite[Proposition 6.2]{lub} points out that the only lattices in a simple
Lie group having deficiency at least two are those in $PSL_2(\R)$. Then
the papers \cite{ltt} and \cite{hll}
classify all groups $G$ of
positive deficiency which are lattices in connected Lie groups:
it turns out that $G$ is a lattice in one of 
$PSL_2(\R)$, $PSL_2(\Cc)$ (which are both rank 1), or
$\R\times PSL(2,\R)$, $\R$ or $\R\times\R$ (which have a soluble factor and
so do not count as semisimple).

But what if $G$ has a presentation of deficiency 1? Here we certainly do
need the presence of finite index subgroups in Definition \ref{rk1pr}. For
instance the fundamental group $G$ of any hyperbolic knot complement
in $S^3$, which is a lattice in $PSL_2(\Cc)$ having a deficiency 1
presentation, satisfies $G/G'=\Z$ and so the pro-$p$ completion $G_{\hat{p}}$
will be the $p$-adic integers $\Z_{\hat{p}}$, topologically generated by
only one element.

However if we are prepared to look at finite index subgroups of a group
having a presentation of deficiency 1 then 
we can detect the rank 1 phenomenon of Definition \ref{rk1pr}
in a very direct way.
(This idea is based on \cite{lub} which deals with lattices in $SL(2,\Cc)$.
     See also \cite[Theorem 5.1]{lac} for a similar result on
     finitely generated groups $G$ in terms of the rank of both
     $H_1(G,\F_p)$ and $H_2(G,\F_p)$.) In Section 5 we give a brief
     explanation of the Reidemeister - Schreier rewriting process, where it
     is shown in Lemma \ref{d1fi} that if $G$ has a presentation of
     deficiency 1 then the process will give us a deficiency 1 presentation
     for $H$ also.
     
     \begin{thm} \label{padan}
       Suppose  we have a group $G$
       which is defined by a presentation $\langle X\,|\,R\rangle$
       of deficiency $|X|-|R|$ equal to 1. Then the following are
       equivalent:\\
              (i) There exists a finite index subgroup $H$ of $G$ with
       $d(H/H')\gqqs 3$.\\
       (ii) There exists a finite
       index subgroup $L$ of $G$ and a prime $p$ such that the pro-$p$
       completion $L_{\hat{p}}$ is not $p$-adic analytic.\\
       (iii) There exists a prime $p$ such that the 
       supremum of rank $H_1(S,\F_p)$ over all finite index subgroups
       $M$ of $G$ is infinite.\\
\end{thm}
       \begin{proof}
         Suppose (i) holds then $H$ will also have a presentation of
         deficiency 1 by Lemma \ref{d1fi}, but now with
         $d(H/H')\gqqs 3$. So Corollary \ref{ineq} applied to $H$
tells us that
$H_{\hat{p}}$ will not be $p$-adic analytic for any prime $p$
such that $d_p\gqqs 3$ and there will be such a prime.

      If (ii) holds then $L_{\hat{p}}$ has infinite Pr\"ufer rank. This means
      that the number of topological generators needed for an open subgroup
      $P$ of $L_{\hat{p}}$ is unbounded as we vary over all such $P$.
      But each $P$ corresponds to an abstract
      finite index subgroup $M$ of $L$ which is subnormal in $L$ and with
      the rank of $H_1(M,\F_p)$ equal to that of $H_1(P,\F_p)$, thus our
supremum is infinite.

Then it is clear that (iii) implies (i), because rank
$H_1(M,\F_p)\lqqs d(M/M')$.
\end{proof}

     Thus we can think of a group $G$ in Theorem \ref{padan}
     as satisfying a different rank 1 phenomenon to that of having
     infinite dimensional bounded second cohomology from the last section.
We now formalise this rank 1 phenomenon in the case of
     deficiency 1 presentations by using the following definition.

     \begin{defn} \label{vsa}
       A group $G$ defined by a deficiency 1 presentation
       has {\bf virtually sufficient abelianisation} if there is a finite index
       subgroup $H$ with $d(H/H')\gqqs 3$.
     \end{defn}

     Certainly not all deficiency 1 groups will have virtually sufficient
     abelianisation, e.g. $\Z$ and $\Z\times\Z$. Indeed $BS(m,n)$ has
     virtually sufficient abelianisation if and only if
     $m$ and $n$ are not coprime,
     for instance see \cite[Section 4]{me1r}
     where many other examples of 2-generator 1-relator
     groups without virtually sufficient abelianisation are given. (It might be
     noted here that all those other exceptions seem to be acylindrically
     hyperbolic; we return to this point in the last section.)

     We finish this section by explaining how
     rank 1 pro-$p$ behaviour as in Definition \ref{rk1pr} can
     be regarded as a genuine rank 1 phenomenon in most cases.

   First if $G$ is an irreducible lattice of higher rank then Margulis's
   theorem
   tells us that $G$ is arithmetic. In that case we can ask whether $G$ has
   the congruence subgroup property which has been proven in most cases.
   If so then we have by \cite[Proposition
   4.1]{lub}  that for any finite index subgroup $H$ of $G$ and for
   any prime
   $p$, the pro-$p$ completion $H_{\hat{p}}$ is $p$-adic analytic.

   As for the rank 1 case, it is shown in
   \cite[Theorem 5.2 and Remark 5.3]{lub}
   that any lattice $G$ in $SL_2(\Cc)$ contains a finite index subgroup $H$
   with $H_{\hat{2}}$ not 2-adic analytic. For the other rank 1 groups, first
   let us consider the general situation where a finitely generated group
   $G$ has a finite index subgroup $H$ which surjects to another group
   $L$ where for some prime $p$, $L_{\hat{p}}$ does not have finite Pr\"ufer
   rank. Then $H_{\hat{p}}$ continuously surjects to $L_{\hat{p}}$, so
   the closed subgroups of $L_{\hat{p}}$ having an arbitrarily high
   number of generators pull back to closed subgroups of $H_{\hat{p}}$
   with the same property. Hence $H_{\hat{p}}$ is not $p$-adic analytic
   and so $G$ satisfies Definition \ref{rk1pr}.
   In particular, if
   $G$ is large so that a finite index subgroup $H$ surjects to $F_2$
   then this occurs (for all primes $p$).

   Largeness was shown in \cite{bhw} for most known lattices in
   real hyperbolic space of higher dimensions. As for complex
   hyperbolic space, some lattices in $PU(2,1)$
   are shown to be large in \cite{deraux}.

   This leaves the exceptional
   quaternionic and octonionic lattices where we do not know if
   they satisfy Definition \ref{rk1pr}. However they cannot be large because
   they have property (T). Consequently they might be thought of as
   rank 1 lattices that ``do not always exhibit rank 1 behaviour''.
   For more on properties of these lattices, see \cite{lbsp}.

     \section{Free by cyclic groups}

     We saw in Corollary \ref{toz} that the only candidates for positive
     deficiency groups $G$ which
     do not have infinite dimensional second bounded cohomology are
     those which are an ascending HNN extension of a finitely generated
     group. This can be divided into two cases, where the ascending HNN
     extension is strict which we deal with in the next section, and
     otherwise where $G$ is of the form $K\rtimes\Z$ for $K$
     the finitely generated kernel of our map to $\Z$ in Corollary \ref{toz}.
     (If the first Betti number of $G$ is at least 2 then it is possible for
     $G$ to be in both cases because $G$ will split in various ways as an HNN
extension. However we would then apply the stronger results
     in this section to $G$.)
     
Here we define a free by cyclic group as a group of the form
$F_r\rtimes\Z$ where $F_r$ is a free group of finite rank $r$. We
allow any $r$, including $\Z$ for $r=0$, as well as
$\Z\times\Z$ and the Klein bottle
group for $r=1$. These last three exceptional groups do not satisfy either of
our two rank 1 phenomena but this is due to them being ``small'' groups.
All free by cyclic groups have a deficiency 1 presentation. Conversely
a result in \cite{dss} tells us that if $G$ is of the form
$K\rtimes\Z$ for $K$ finitely generated and $G$ has a deficiency 1
presentation then $K$ is free. If $r\geq 2$ then there are various ways to
see that $G$ itself has both of our rank 1 properties.
Here we proceed
by using mainly acylindrical hyperbolicity for the bounded cohomology
and a short argument for the virtually sufficient abelianisation. 

\begin{thm} \label{fbyz}
  Any free by cyclic group $G=F_r\rtimes \Z$ with $r\geq 2$
  has infinite dimensional second bounded cohomology and virtually
  sufficient abelianisation.
\end{thm}  
\begin{proof} For the first part, we are done by Proposition \ref{inf2} (iii)
  if $G$ is acylindrically hyperbolic. By \cite{gh}
  we will find that this holds
  unless the defining automorphism $\alpha$ of $F_r$ has finite order
  in $Out(F_r)$ whereupon $G$ is virtually $F_r\times \Z$. 
In this case $G$ is definitely not acylindrically
  hyperbolic, as if it were then $F_r\times\Z$ would be too.
However we can argue that by \cite[Proposition 4.1]{lvrf} we see $G$ is
a generalised Baumslag - Solitar group.
Thus by Theorem \ref{gbs} (ii) and the fact that $G$ is not soluble, we
conclude that $G$ has infinite dimensional second bounded cohomology.

For virtually sufficient abelianisation, we use the trick
(as in \cite[Proposition 8.7]{lac} or \cite[Lemma 3.6]{mearx})
that $\alpha$ gives rise to the abelianised automorphism
$\overline{\alpha}$ of $\Z^r$ and hence to the mod $p$ abelianised
automorphism $\overline{\alpha}_p$ of $\F_p^r$. As the mod $p$ automorphism
has finite order $d$ say, we can take the $d$-fold cyclic cover $H$ of $G$,
which is $\langle t^d,F_r\rangle$ where $t$ is the stable letter
inducing $\alpha$ by conjugation. Then we have that
$(\overline{\alpha}_p)^d=\overline{(\alpha^d)}_p$ is the identity. But $H$
has index $d$ in $G$ and the rank of $H_1(H,\F_p)$ is $r+1$, with the 1
coming from the $\Z$ in the abelianisation of $H$ given by the stable
letter $t^d$. As $r+1\geq 3$, we have virtually sufficient abelianisation
and so $G$ satisfies Definition \ref{rk1pr} by Theorem \ref{padan}.
Indeed as our argument here works for every prime, we also have that $G$
satisfies the stronger Definition \ref{rk1p}.
\end{proof}

\section{Strictly ascending HNN extensions}

We have now considered both of our rank 1 phenomena for all presentations
of positive
deficiency, except for when the presentation has deficiency
exactly 1 and the resulting group $G$
is an HNN extension $\langle t,A\rangle$
where $t$ is the stable letter, the vertex group $A$ is finitely
generated and the HNN extension is formed from an endomorphism
$\theta:A \rightarrow A$ which is injective but not surjective. The obvious
case resulting in a deficiency 1 presentation is
when $A$ is a free group $F_r$ of rank at least 1. It is unknown if there
are other finitely generated groups $A$ that can occur here but the conjecture
after Remark 2.2 in \cite{bi} is that it does not occur. Moreover 
Theorem A (c) in that paper states that $G=A*_\theta$ will have cohomological
dimension (at most) two.
 We now apply \cite[Corollary 2.5]{gmsw} which states the following.
  Let $A$ be
  of type $FP_2$ with cohomological dimension 2. Let $\theta:A\rightarrow A$
  be a monomorphism and let $G$ denote $A*_\phi$. Then the cohomological
  dimension of $G$ is 3. Hence any group $A$ which could occur in such
  a counterexample would not be $FP_2$.

  We now assume for the rest of this section that $A$ is the free group $F_r$.
For $r=1$ we have
the Baumslag - Solitar groups $BS(1,n)$
where $|n|\gqqs 2$, which we have seen are soluble and therefore too ``small''
to satisfy either of our rank 1 phenomena.

As for a general ascending HNN extension of $F_r$,
we have:
\begin{prop} \label{quot}
If $G=F_r*_\theta$ is an ascending HNN extension of a finite
  rank free group then:\\
  (i) (\cite[Proposition 4.4 (i)]{mearx})
  $G$ cannot contain a Baumslag -- Solitar subgroup $BS(m,n)$ (where
without loss of generality $m\gqqs 1$ and $|n|\gqqs m$)
 unless $m=1$ or $n=\pm m$.\\
 (ii) (\cite[Proposition 4.4 (iii)]{mearx})
 For $d\neq 0$, $G$ contains $BS(1,d)$ if and only if we have an
  element $w\in F_r\setminus\{e\}$ and $i>0$ with $\theta^i(w)$ conjugate
  in $F_r$ to $w^d$. In this case,
  we call (the conjugacy class of) $w$ a {\bf periodic
    conjugacy class} of $\theta$.
\end{prop}

Note that $BS(m,\pm m)$ contains $\Z\times\Z$ anyway, so $G$ containing
a Baumslag - Solitar group is equivalent to $G$ containing
$BS(1,d)$ for some $d\neq 0$, which in turn is equivalent to $\theta$
containing a periodic conjugacy class.

Of course if a group is word hyperbolic then it cannot contain any
Baumslag - Solitar subgroup.
In \cite{mut} the converse was shown for strictly ascending HNN extensions,
so that if $G=F_r*_\theta$ is not word hyperbolic then $\theta$ has a
periodic conjugacy class.
Thus we do have infinite dimensional second bounded cohomology for $G$ if 
there are no periodic conjugacy classes for $\theta$.

Strictly ascending HNN extensions of finite rank free groups
which contain Baumslag - Solitar groups can be somewhat mysterious and
not as well behaved as free by cyclic groups in general. Some
properties still hold good, for instance they are all known to be
coherent by \cite{fgnhn} and residually finite by \cite{bs}.
However we have the following examples.

\begin{ex} \label{drsp} 
  Following \cite{drusap}, given $r\gqqs 1$ and $d\neq 0$ we define
  the {\bf Dru\c{t}u - Sapir group} $DS(r,d)$ to be the group
  with presentation
  \[\langle t,a_1,\ldots ,a_r\,|\,ta_1t^{-1}=a_1^d, \ldots ,ta_rt^{-1}=a_r^d
    \rangle\]
  which is clearly an ascending HNN extension of $F_r$ and is strictly
  ascending if and only if $|d|\gqqs 2$. If $|d|=1$ or $r=1$ then it is
  a linear group but it is shown in \cite{drusap} that otherwise it is
  non-linear. Moreover consider the ascending HNN extension $C(r,d)$
  of $F_r$ with presentation
  \[\langle t,a_1,\ldots ,a_r\,|\,ta_1t^{-1}=a_2, \ldots ,
    ta_{r-1}t^{-1}=a_r,ta_rt^{-1}=a_1^d
    \rangle.\]
  Note that $C(r,d)$ has a 2-generator 1-relator presentation, given by
  eliminating $a_r$ then $a_{r-1}$ and so on, right up to $a_2$. But the
  obvious cyclic cover of $C(r,d)$ having index $r$ is $DS(r,d)$, so
  that $C(r,d)$ is also not linear if $r\gqqs 2$ and $|d|\gqqs 2$.
\end{ex}

However other ascending HNN extensions of $F_r$ for $r\gqqs 2$
can contain non - Euclidean Baumslag - Solitar groups but still be
linear, for instance in \cite{caldun} where in $SL(2,\Cc)$ they take
$t$ and $F_6=\langle a\rangle *F_5$, with $t$ conjugating $F_6$ into itself
such that $tat^{-1}=a^3$.

We now turn to asking which strictly ascending HNN extensions $G$ of finitely
generated free groups have virtually sufficient abelianisation. Note that the
argument used in the last section for free by cyclic groups will not work in
general now. Indeed if the defining endomorphism $\theta$ sends every
generator of $F_r$ into the commutator subgroup $F_r'$ then the induced
mod $p$ abelianised endomorphism of $\F_p^r$ will be zero for every prime $p$.
On the other hand, if $G$ is word
hyperbolic and the defining endomorphism $\theta$ is irreducible then
\cite{hgws} shows that $G$ acts
freely and cocompactly on a CAT(0) cube complex. By the full force of the
Agol - Wise machinery (see \cite{ag} and \cite{wsbgbk}), this implies that
$G$ is virtually the fundamental group of a special cube complex and so
will virtually embed in a RAAG by \cite{hglws}. Then any non-abelian
subgroup of a RAAG surjects to $F_2$
(see \cite[Corollary 1.6]{igash}) so we conclude that $G$ is
large. But we do not know of a more elementary
proof of virtually sufficient abelianisation for these groups.
Furthermore if $G$ is hyperbolic but $\theta$ is reducible then $G$ is not
known to be cubulated. Here we do not know of an argument that shows
$G$ has virtually sufficient abelianisation, although hyperbolicity means $G$
satisfies our other rank 1 phenomenon.

\medskip
We proceed by first reviewing some consequences of the 
Reidemeister - Schreier rewriting process which is itself described
in \cite[Proposition II.4.1]{ls}.

Suppose we have a group $G$
given by the presentation $\langle X\,|\,R\rangle$
(not assumed finite here) and we want to get a presentation for a subgroup
$H$ (not assumed to be finite index)
of $G$. This can be done by taking
a Schreier transversal $S$ for $H$ in $G$ which
is a right transversal for $H$ in $G$, given as words in $X$,
having the property that any initial subword of $s\in S$
is also in $S$. Note that (unless $S=\{e\}$ in which case $H=G$) some element
of $X$ or its inverse must appear in $S$. For any $g\in G$, let $\overline{g}$ 
be the representative of $g$ in $S$, that is the unique $s\in S$ such that
$Hs=Hg$. We then have that
\[Y:=\{sx(\overline{sx})^{-1}\,|\,s\in S,x\in X\}\]
is (throwing away occurrences of the identity) a generating set
for $H$. As for obtaining a set of relators for $H$,
we can do this by considering for $s\in S$ and $r\in R$ the relator
$srs^{-1}$ which can then be
written in terms of the elements of $Y$ to obtain a relator
in $Y$. Doing this over all $s\in S$ and $r\in R$ gives us a defining
presentation for $H$ with these relators and our generating set $Y$.

We record a useful consequence of this process which was used in Section 3.
\begin{lemma} \label{d1fi}
  Suppose that $G$ is a group given by a presentation of deficiency 1.
  Then any finite index subgroup $H$ of $G$ also possesses a
  presentation of deficiency 1.
\end{lemma}
\begin{proof}
First suppose more generally that $G$ is given by the finite  
presentation $\langle X\,|\,R\rangle$
where $|X|=m$ and $|R|=n$. Let $H$ be a subgroup of index $i$.
Then the above process provides $im$ generators and
$in$ relators in the resulting presentation for $H$. However if $sx=
\overline{sx}$ then we obtain the identity which can be thrown away. As
$S$ is a Schreier transversal, this will happen (at least) $i-1$ times
so that we actually have $i(m-1)+1$ generators and $in$ relators in the
corresponding presentation for $H$. This gives us our result if
$m-n=1$. 
\end{proof}

We now use this process when $G$ is a free group. Here
we say that $g\in G$ is {\bf primitive} in $G$ if it is an element of some
free basis for $G$.

\begin{lemma} \label{prim}
  Let $G$ be a free group and $g$ be a primitive element of $G$. Suppose
  we have a subgroup $H$ of $G$ and $k>0$ such that $g^k\in H$ but
  $g,g^2,\ldots ,g^{k-1}$ are not in $H$. Then $g^k$ is a primitive element
  of $H$.
\end{lemma}
\begin{proof}
Let $X$ be a free basis of $G$ in which $g$ appears. Then for
$0\lqqs i<j<k$ we have $Hg^i\neq Hg^j$ so that $e,g,\ldots, g^{k-1}$ is a
partial Schreier transversal. On extending to a full Schreier transversal $S$,
we will find elements of our generating set for $H$ by taking the various
expressions $sx(\overline{sx})^{-1}$ which are not equal to the identity.
When we set $x=g$, we can also take $s=e,g,\ldots ,g^{k-2},g^{k-1}$ in turn.
Whereas $e,\ldots ,g^{k-2}$ all return $e$, for $s=g^{k-1}$ we find that
$g^k\neq e$ is an element of our generating set. To see that this
generating set $Y$ is a free
generating set, note that the original set of relators $R$ for $G$ is empty.
Thus the new set for $H$ using the generating set $Y$ is too.
\end{proof}

So now suppose we have $G$ a strictly ascending HNN extension of the free
group $F_n$ (where $n\gqqs 2$) which contains some subgroup $BS(1,d)$ for
$d\neq 0$. We said above that this gives rise to a
periodic conjugacy class, so that there is $i>0$ and
  $w,v\in F_n$ where $w$ is not the identity with $\theta^i(w)=vw^dv^{-1}$.
  Note that if $w=u^k$ is a power of some other element $u\in F$ then
  $(\theta^i(u))^k=(vu^dv^{-1})^k$ and $k$th roots are unique in a free
  group, so we can replace $w$ with $u$ and hence
  assume that $w$ is not a proper power.

  We will now tidy this up at the cost of replacing $G$ with a finite
index subgroup $H$, which will not matter because we allow this in our
virtually sufficient abelianisation condition from Definition \ref{vsa}. 
We first replace $\theta$ with $\phi=\theta^i$
  which just changes $G$ to its finite cyclic cover of index $i$
  (which is still an ascending HNN extension).
  We can then replace $\phi$ with $\iota_{v^{-1}}\phi$ for $\iota_{v^{-1}}$
  the inner automorphism which is conjugation by $v^{-1}$ and which does
  not change the group. Renaming
  $\iota_{v^{-1}}\phi$ by $\theta$ and the cyclic cover of $G$ by
  $G$ also, we now have $\theta(w)=w^d$.

  We now show that when $|d|\gqqs 2$ (though not necessarily when $d=\pm 1$),
  this element $w$ is primitive in $F_n$. This result is in
  a paper \cite{lgv4}
  of Logan which uses the machinery of Mutanguha in \cite{mut}. Here we
  adapt the more basic idea from Lemma 3.4 in an earlier version (arvix
  version 3) of Logan's paper.
  \begin{thm} \label{melg}
    Let $\theta: F_n\rightarrow F_n$ be an injective endomorphism of the free
    group $F_n$ for $n\gqqs 2$ and let $w\in F_n\setminus\{e\}$ be such
    that $\theta(w)=w^d$ for $|d|\gqqs 2$ but $w$ is not a proper power.
    Then $w$ is primitive in $F_n$.
  \end{thm}
  \begin{proof}
    An old and classic theorem of G.\,Baumslag in \cite{bmrn} is as follows.
    Say $g_1,\ldots ,g_n$ are elements of the free group $F$ on
    $x_1\ldots ,x_n$ for $n\gqqs 2$
    and $w(x_1,\ldots ,x_n)$ is a word which is not a proper power and not
    primitive when regarded as an element of $F$.
    If we further have $g\in F$ and
    $k\geq 2$ such that the relation $w(g_1,\ldots ,g_n)=g^k$ holds in $F$
    then the subgroup $S:=\langle g_1,\ldots ,g_n,g\rangle$, which must
    clearly have rank at most $n$, actually has rank less than $n$.

    Let us set $g_i=\theta(x_i)$ for $1\lqqs i\lqqs n$.
    We have a word $w\in F_n\setminus\{e\}$ with $\theta(w)=w^d$ so we can
    put $g=w(x_1,\ldots ,x_n)^{\pm 1}$ (according to whether $d$ is positive
    or negative) and $k=|d|$ to obtain
    \[\theta(w(x_1,\ldots ,x_n))=w(\theta(x_1),\ldots ,\theta(x_n))
      =w(g_1,\ldots ,g_n)=g^k.\]
    As we said that $w$ is not a proper power, we have that $w$ is primitive
    by Baumslag's result unless
    $\langle \theta(x_1),\ldots ,\theta(x_n),w\rangle$ has rank less than $n$.
    Now the subgroup $\langle \theta(x_1),\ldots ,\theta(x_n)\rangle$ of $F$ has
    rank $n$ because $\theta$ is injective. However it can happen that 
    throwing in $w$ can reduce the rank. But notice that if $n=2$
    then we would be done, because  $\langle \theta(x_1),\theta(x_2)\rangle$
    is non-abelian, so cannot be made cyclic by adding an extra generator.

    Therefore we invoke a rank reduction argument. The key here is that
    if we have the strictly ascending HNN extension $\langle t,F\rangle$
    where $F$ is finite rank free, defined by the endomorphism $\theta$,
    and $S$ is a finitely generated subgroup of $F$ (not necessarily
    of finite index) of rank $m$ say
    with $\theta(F)\lqqs S\lqqs F$ then
    $\langle t,F\rangle$ is also naturally the ascending HNN extension
    $\langle t,S \rangle$, because $\theta(S)\lqqs \theta(F)\lqqs S$, so
    we just take a free basis $y_1,\ldots ,y_m$ for $S$ and express each
    $\theta(y_i)$ in terms of this basis.

    Here we will set $S:=\langle \theta(F),w\rangle$ which is clearly
    finitely generated and contains $\theta(F)$. We saw above that if
    the rank of $S$ is equal to that of $F$ then we are done. But
    otherwise the rank is smaller, so we can replace $F$ with $S$ in
    the above and continue since $S$ also contains $w$,
    with the process terminating at some point
    since we cannot go below rank 2.

    However we are changing the vertex group of the extension each time.
    Suppose our rank does reduce when going from $F$ to $S$ but not at the
    next step. Then our application of Baumslag's result is telling us that
    $w$ is primitive in $S$, not necessarily $F$. But supposing that the
    former case holds (and it will hold after finitely many reductions),
        we are done if we can show that $\theta(w)$ is
    primitive in $\theta(F)$, as we then use injectivity of $\theta$.
    Now $\theta(w)=w^d\in\theta(F)$ and as $k=|d|$,
    we will be done by Lemma \ref{prim} on setting
    $G=S$, $g=w$ which is primitive in $G$
    and $H=\theta(F)$ provided $w,\ldots ,w^{k-1}\notin \theta(F)$. So suppose
    we have $u\in F$ such that $\theta(u)=w^j$ where $1\lqqs j<k$. Then
    $\theta(u^k)=w^{jk}=\theta(w^{\pm j})$. By injectivity this would say that
    $u^k=w^{\pm j}$ with $u\notin \langle w\rangle$ so that $w$ is a proper
    power, a contradiction.
    \end{proof}
    
    \begin{thm} \label{main}
      Let $\theta:F_r\rightarrow F_r$ be an injective endomorphism
  of the free group $F_r$ where $r\gqqs 2$ and let
  $G=\langle t,F_r\rangle$ be
  the corresponding HNN extension. If $G$ is not word hyperbolic
  then there is
  a finite index subgroup $H$ of $G$ and a prime $p$ such that the rank
  $H_1(H,\F_p)$ is at least 3.
\end{thm}
\begin{proof}
The free by cyclic case is covered in the previous section, so
we can assume we have a strictly ascending HNN extension. Then
by \cite{mut}, if $G$ is not word hyperbolic it will contain
$BS(1,d)$ for $d\neq 0$.

Let us first deal with $d=\pm 1$.
We can assume that $d=1$ because $BS(1,-1)$ contains $BS(1,1)$.
Then by \cite[Corollary 4.6]{med1} we have that $G$ is large and
so has virtually sufficient abelianisation.

Otherwise by Proposition \ref{quot} (ii) and the tidying up
that was undertaken after Lemma \ref{prim},
we have $w\in F_r\setminus\{e\}$ with $\theta(w)=w^d$.
As $|d|\gqqs 2$ now,
Theorem \ref{melg} says that $w$ is
actually a primitive element of our free group.
  Thus on taking a free basis $t,a,a_2,\ldots ,a_r$ for $F_r$ with $w=a$ as its
  first element, we obtain for $G$ the presentation
\[G=\langle t,a,a_2, \ldots ,a_r\,|\,tat^{-1}=a^d,
  ta_2t^{-1}=w_2,\ldots ,ta_rt^{-1}=
  w_r\rangle\]
where $|d|\gqqs 2$ and $w_2,\ldots ,w_r$ are elements of $F_r$ that will not
particularly concern us.
The abelianisation of $G$ maps onto $\Z$ by sending
$t$ to 1 and all other generators to zero. Moreover by abelianising
the first relation, we obtain $(d-1)\underline{a}=0$ where we write
$\underline{\cdot}$ for the image of $\cdot$ when we abelianise.
We now pick any prime $p$ dividing $d-1$. If $d=2$ then there are
no such $p$ but we can take the double cyclic cover and on replacing
$t^2$ with $t$, our first relation is now $tat^{-1}=a^4$, hence in
this case we then choose $p=3$.

Next we consider the quotient group $G/G'G^p$. This is
an finite elementary $p$-group so is of the form $C_p^i$ for $i$ equal
to the rank of $H_1(G,\F_p)$. But we can treat such a group as a vector
space of dimension $i$ over $\F_p$ where we regard the $r+1$ generators
and $r$ relators as defining $G/G'G^p$ as a quotient vector space
of $\F_p^{r+1}$, given by the $r$ equations which are obtained from the relators
by abelianising, which we write additively, and taking the coefficients
modulo $p$. But as $p$ divides $d-1$, we have that the first equation
$(d-1)\underline{a}$ vanishes altogether when now taken modulo $p$.  

Consequently we have $r-1$ equations so $G/G'G^p$ must have rank at
least 2. Now given any index $i$ subgroup $H$ of $G$, we
mentioned that we can obtain a
presentation of $H$ by the Reidemeister - Schreier rewriting process
described above, which in this case will have
$ir+1$ generators and $ir$ relators.

We now proceed as follows. We assume that we can find some
proper finite index subgroup $H$ such that both $t$ and
$a$ are in $H$. As
$x=t$ and $x=a$ were part of our generating set $X$ above for $G$,
on taking the identity $e$ for our first element $s$ in our Schreier
transversal, we obtain the corresponding elements
$x(\overline{x})^{-1}$ in our generating set for $H$, 
two of which will be $t$ and $a$. But there must exist some other generator
$a_2,\ldots ,a_r$ for $G$, say $b$, with $b^{\pm 1}$ in $S$ 
because $S$ must contain a length
one element (or $S=\{e\}$ but that is when $H=G$). On now taking
$x=t$ and $x=a$ again but $s=b^{\pm 1}$, we find that
$sx(\overline{sx})^{-1}$ is equal to $t':=b^{\pm 1}tb^{\mp 1}$ and
$a':=b^{\pm 1} ab^{\mp 1}$.
Having obtained all of the $ir+1$ generators for $H$ including these four, we
form the $ir$ relators for $H$ by rewriting the original relators and their
conjugates by elements of $S$ in terms of our new generating set $Y$
for $H$. But our first relator $tat^{-1}a^{-d}$ stays as it is because
$t,a\in Y$. Since $b^{\pm 1}\in S$,
we next conjugate this relator by $b^{\pm 1}$ to obtain
$b^{\pm 1}tat^{-1}a^{-d}b^{\mp 1}$
which will be rewritten as $t'a't'^{-1}a'^{-d}$. It does not now matter
about the other relators, because when we now abelianise our deficiency 1
presentation for $H$ and work modulo $p$ to obtain our equations over
$\F_p$, these two equations will vanish. Consequently $H/H'H^p$ is the
quotient of $\F_p^{ir+1}$ by at most $ir-2$ equations and so 
the rank of $H_1(H,\F_p)$ is at least 3.

We now have to ensure that such an $H$ exists. It certainly will not
if $t,a$ generate $G$. But as we have the relation $tat^{-1}=a^d$
between them, if they did then $G$ would be a quotient of $BS(1,d)$
and so would be metabelian, which is a contradiction as $G$ contains
$F_r$ for $r\geq 2$. So we can hope that there is a finite quotient $G/N$
of $G$ which is not generated by the images $tN$ and $aN$ in $G/N$. Then
$\langle tN,aN\rangle$ will be a proper subgroup of $G/N$ and by
correspondence this will take the form $H/N$, where
$H$ will be a proper finite index subgroup containing $t$ and $a$ so
that we would be done. To obtain such a finite quotient, recall by
\cite{bs} that $G$ is residually finite. Also $G$ is not metabelian, so
that we can take $\gamma \in G''\setminus \{e\}$ and a finite quotient
$Q$ of $G$ in
which $\gamma$ survives. Thus $Q''\neq \{e\}$, but if $Q$ were generated
by the images of $t$ and $a$ then it would be a quotient of
$\langle t,a\rangle$ which is metabelian, so $Q$ would be too and thus
$Q''=\{e\}$.
\end{proof}

We finish this section by pointing out that for word hyperbolic
strictly ascending HNN extensions of finite rank free groups,
if the endomorphism $\theta$
is reducible then it is not known whether the resulting group
has virtually sufficient abelianisation. 

\section{Bounded Generation}

We now mention some applications of our results from the previous sections.
The first is to bounded generation.

\begin{defn} \label{bgdf}
  A group $G$ is {\bf boundedly generated}
  if there exists a finite (generating) set
  $c_1,\ldots ,c_n\in G$ and $N\in\N$ such that
  every $g\in G$ can be expressed as
  \[g=c_{i_1}^{m_1}\ldots c_{i_N}^{m_N}\mbox{ where }
    1\lqqs i_1,\ldots ,i_N\lqqs n\mbox{ and } m_1,\ldots ,m_N\in\Z.\]
\end{defn}
Another frequently used definition of bounded generation is that $G$
is equal to the product $C_1\ldots C_M$ where
$C_1,\ldots ,C_M$ are cyclic subgroups of $G$ (possibly with repeats).
This clearly implies Definition \ref{bgdf}. For the other way round, note
that if $G$ is boundedly generated by $c_1,\ldots ,c_n$ with word length
$N$ then we can
express $G$ as the product of cyclic groups by making
sure that this product
contains every length $N$ word in $\{1,\ldots ,n\}$. For instance,
if $G$ is boundedly generated by $a,b$ with $N=3$ in Definition \ref{bgdf}
then for
$A=\langle a\rangle$, $B=\langle b\rangle$ we would have
$G=AAABABBBAA$, so that in general we should expect $M$ to be much larger
than $N$.

Obvious examples of finitely generated groups which are boundedly generated
comprise finite groups (but no other torsion groups as there will only be
finitely many possibilities for $g$ in Definition \ref{bgdf}), alongside
abelian groups.
It is also the case that if $G$ is boundedly generated then so are the
finite index subgroups and supergroups of $G$. Moreover bounded generation
is preserved by taking quotients and indeed extensions (meaning that
if $N\unlhds G$ with $N$ and $G/N$ both boundedly generated then so is $G$).
So for instance this means that all virtually polycyclic groups are
boundedly generated.

Slightly less obvious examples can be formed in the
following way. If $G=C_1\ldots C_M$ is boundedly generated and
$\theta:G\rightarrow G$ is an injective (but not necessarily surjective)
endomorphism then on forming the ascending HNN extension
\[\Gamma=\langle G,t\,|\,tgt^{-1}=\theta(g)\mbox{ for }g\in G\rangle,\]
we can move $t$ to the right of any $g$ and $t^{-1}$ to the left,
so that any element in $\Gamma$ can be written as $t^{-p}gt^q$
for $g\in G$ and $p,q\gqqs 0$. Hence if $T=\langle t\rangle$ then
$\Gamma=TC_1\ldots C_MT$ is also boundedly generated. In particular
$BS(1,n)$ is boundedly generated. A much less obvious example
is given in the paper \cite{mur} of Muranov. This gives an example of a
finitely generated (though not finitely presented) group which contains
$F_2$ and is a direct limit of a sequence of hyperbolic groups with
respect to a family of surjective homomorphisms.

The other main source of boundedly generated groups are some
$S$-arithmetic subgroups of linear algebraic groups. In particular
it was shown in \cite{crtkl83} that 
for the ring of integers $\mathcal{O}$ in a number field,
$SL_n(\mathcal{O})$ is boundedly generated for $n\geq 3$ using
the elementary matrices. Unpublished work of Carter--Keller--Paige
extended this to $n=2$ when $\mathcal{O}$ has infinitely many units
(see \cite{wttmrr} for an account using first-order logic). This
case was recently extended further to $S$-integers
$\mathcal{O}_S$ of a global field (in any characteristic) in \cite{mrsmrg}.
However
$SL_2(\Z)$ is not boundedly generated as it is virtually non-abelian
free. Thus in looking
for groups which are not boundedly generated, we might hope that
either of our two rank 1 phenomena are enough to establish this. Indeed
this is well known, as we now summarise:
\begin{prop} \label{ntbg}
  Let $G$ be a finitely generated group.\\
  (i) If $G$ has an infinite dimensional
  space $Q_h(G)$ of homogeneous quasimorphisms then $G$ is not
  boundedly generated. Thus if $G$ is finitely presented and has infinite
  dimensional second bounded cohomology then $G$ is not boundedly
  generated.\\
  (ii) If $G$ is a finitely generated group having a finite index subgroup
  $H$ whose pro-$p$ completion $H_{\hat{p}}$ for some prime $p$ is not
  $p$-adic analytic then $G$ is not boundedly generated.
\end{prop}
\begin{proof}
  The point in (i) is that if $G$ is boundedly generated by $g_1,\ldots ,g_n$
then any homogeneous quasimorphism which is zero on all of
$g_1,\ldots ,g_n$ must be bounded on $G$ and hence
identically zero, by homogeneity. But if we have
$n+1$ linearly independent homogeneous quasimorphisms then we can solve
$n$ linear equations to find a non-trivial homogeneous quasimorphism which
is zero at $g_1,\ldots ,g_n$. Then
$G$ being finitely presented implies that $H^2(G;\R)$ is finite dimensional,
so $H^2_b(G;\R)$ being infinite dimensional is equivalent to $Q_h(G)$
being infinite dimensional.

For (ii), if $G$ were boundedly generated then $H$ would be too, but
the pro-$p$ completion $H_{\hat{p}}$ of a boundedly generated abstract group
$H$ must also be boundedly generated (where we now mean topological
generation). But a boundedly generated pro-$p$ group must be $p$-adic
analytic, say by \cite[Corollary 12.6.2]{lssg}.
\end{proof}

We can now apply all of this to groups with a positive deficiency presentation.
\begin{thm} \label{psdfthm}
  If $G$ is a group defined by a presentation of positive deficiency then\\
$\bullet$ either
  $G$ is isomorphic to $BS(1,n)$ for $n\in\Z$\\ 
$\bullet$ or $G$ satisfies at least one of rank 1 bounded cohomology behaviour
or rank 1 pro prime behaviour\\
$\bullet$ or $G$ is a strictly ascending HNN extension
  of the finitely generated group $H$ where $H$ is not $FP_2$.
\end{thm}
Conjecturally in \cite{bi}
  this last case does not occur.
\begin{proof}
  There will exist a homomorphism from $G$ onto $\Z$
  which means we can express $G$ as an HNN extension $A*_\phi$.
  We saw in Corollary
  \ref{toz} that we can assume the vertex group $A$ and edge group
  $C$ are finitely generated. If neither $C$ nor $\phi(C)$ are equal
  to $A$ then that corollary tells us we have infinite dimensional
second bounded cohomology.
 If $C=\phi(C)=A$ then $G=A\rtimes\Z$ and we mentioned in
  Section 4 that $A$ will be free, so Theorem \ref{fbyz}
tells us that we have virtually sufficient abelianisation, 
with the only exceptions being when $G=\Z$,
  $\Z\times\Z$ or the Klein bottle group
  which are all boundedly generated.

  Thus we now know that $C$ is equal to $A$ but $\phi(C)\neq A$ (or
  vice versa but then just swap $\phi(C)$ and $C$)
  so that $G$ is a strictly ascending HNN extension of the finitely
  generated group $A$. If $A$ is of type $FP_2$ then it is known by
  the comments at the start of Section 5
  that $A$ must be free and it is conjectured at the
  end of \cite[Section 2]{bi} that this is always the case.

  Hence we are left with $G$ being a strictly ascending HNN extension
  of a finite rank free group. If the rank is
  one then $G=BS(1,n)$ which
  is boundedly generated. If the rank is at least two then
  we have infinite dimensional second bounded cohomology
  if $G$ is word hyperbolic. Otherwise we have virtually sufficient
abelianisation by Theorem \ref{main}.
\end{proof}

\begin{cor} \label{psdfbg}
If $G$ is defined by a presentation of positive deficiency then\\
$G$ is not boundedly generated, unless
$G=BS(1,n)$ or $\Z$ (which are boundedly generated) or $G$ is a
strictly ascending HNN extension of the finitely generated group
$H$ where $H$ is not $FP_2$.
\end{cor}
Conjecturally the last case does not occur.
\begin{proof}
This follows immediately from Theorem \ref{psdfthm} and Proposition
\ref{ntbg}.
\end{proof}

Looking back at Proposition \ref{ntbg}, we note here that neither
converse is true. For this, we first define the notion of polynomial
index growth, as explained in \cite[Chapter 12]{lssg}. 
If $G$ is a finitely generated (abstract, profinite
or pro-$p$) group then we can define the index growth function $f_G(n)$
by the supremum of $|Q/Q^{(n)}|$ where $Q$ ranges over all finite quotients
of $G$ and $Q^{(n)}$ is the normal subgroup of $Q$ generated by the $n$th
powers of elements of $Q$. By Zelmanov's resolution of the restricted
Burnside conjecture, $f_G(n)$ is finite and we say that $G$ has
{\bf polynomial index growth} if there exists $C\geq 0$ such that
$f_G(n)\lqqs n^C$. Now if $G=C_1\ldots C_M$ has bounded generation
(in either the abstract,
profinite or pro-$p$ sense) then the size of $|G/G^{(n)}|$,
and hence $Q/Q^n$ for any quotient $Q$ is at most $n^M$
and so $f_G(n)$ is polynomial in $n$.

For a counterexample to the converse of
(a), we can take any finitely generated residually
finite group $G$ that is virtually soluble but which does not have
finite Pr\"ufer rank. (For abstract groups $G$, finite Pr\"ufer rank
means that $d(H)$ is bounded over all finitely generated subgroups
$H$ of $G$.) For instance
we can take the wreath product 
\[\Z \wr \Z=\langle s,a\,|\,[a,s^ias^{-i}] \mbox{ for all }
i\in\Z\rangle\] which contains $\Z^r$ for all $r>0$.
Finitely presented examples exist too, for instance we can use this
wreath product as the
base for an ascending HNN extension
\[\langle a,s,t\,|\,tat^{-1}=asas^{-1},[a,sas^{-1}=1]=[s,t]\rangle\]
which is known as the Baumslag -- Remeslennikov group.
Then such a group $G$ is amenable thus
has no bounded cohomology but \cite{pysg} shows that $G$ does not have
polynomial index growth, so is not boundedly generated. For (b), examples
of finitely generated groups $G$ (linear in characteristic zero) where
$G$ is not boundedly generated but the profinite completion $\hat{G}$
is were given in \cite{invt}. Thus $\hat{G}$ has polynomial index growth
and so does $G$, because the finite quotients of $\hat{G}$ are the same
as those for $G$. But for any prime $p$ the finite quotients of
$G_{\hat{p}}$ are a subset of this, so $G_{\hat{p}}$ also has polynomial
index growth. Now in the world of pro-$p$ groups, having finite Pr\"ufer
rank, being boundedly generated and having polynomial index growth are
the same thing (see \cite{lssg} Chapter 12 and references within). Thus
$G_{\hat{p}}$ is $p$-adic analytic. We can now repeat with $H$ any finite
index subgroup of $G$, so that $\hat{H}$ will be a finite index subgroup
of $\hat{G}$ and thus will also be boundedly  generated.
  
  \section{1-relator groups}

  We finish by specialising to 1-relator groups where we are able to
 strengthen these results.

\begin{lemma} \label{rel1}
  If $G$ is a finitely generated group defined by a
  1-relator presentation then either
  $G$ is cyclic, is
  isomorphic to $BS(1,n)$ for $n\in\Z\setminus\{0\}$ or
  $G$ has infinite dimensional second bounded cohomology
  or $G$ has a finite index subgroup $H$ where
  $H_{\hat{p}}$ is not $p$-adic analytic for some prime $p$.
\end{lemma}
\begin{proof}
  We follow the proof of Theorem \ref{psdfthm}. However a non-cyclic
  1-relator group
  can be written as an HNN extension where not only do we have that
  the vertex group $A$ and edge groups $C,\phi(C)$ are all finitely generated,
  but using the Freiheitssatz we can ensure that $A$ is itself 1-relator and
  $C$ and $\phi(C)$ are free.
  Thus if the HNN extension is ascending then $A$ is free, whereupon the
  results in Sections 4 and 5 apply without the exceptional case
  occurring.
\end{proof}

However we can take this further by looking at the family of subgroups
of 1-relator groups. We have the following facts.
\begin{prop} \label{1rfcts}\hfill\\
  (i) (\cite{lntj}) All 1-relator groups are coherent; that is any finitely
  generated subgroup is finitely presented.\\
  (ii) (\cite{brds} or \cite{hw})
  All torsion free 1-relator groups are locally indicable; that is any
  finitely generated subgroup surjects to $\Z$.\\
  (iii) (\cite{lyd})
  All torsion free 1-relator groups have cohomological dimension two.
\end{prop}

\begin{thm} \label{rel1sb}
  If $G$ is a group generated by a 1-relator presentation and $H$ is any
  finitely generated subgroup of $G$ then
  $H$ is cyclic or $C_2*C_2$ or is
  isomorphic to $BS(1,n)$ for $n\in\Z\setminus\{0\}$, or
  $H$ virtually has infinite dimensional second bounded cohomology
  or $H$ has a finite index subgroup $L$ where
  $L_{\hat{p}}$ is not $p$-adic analytic for some prime $p$.
\end{thm}
\begin{proof}
  First suppose that $G$ is a torsion free 1-relator group.
  We proceed by applying Corollary \ref{toz} to the subgroup
$H$ of $G$. This is
  valid because $H$ surjects to $\Z$ by Proposition \ref{1rfcts} (ii)
  and it tells us that we are done if the corresponding HNN extension
  is non-ascending by second bounded cohomology.

  Now suppose we have
  $H$ as an ascending HNN extension $A*_\phi$ with vertex
  group $A$ and edge groups $C=A$, $\phi(C)\lqqs A$ which we can assume
  are finitely generated and thus finitely presented by
  Proposition \ref{1rfcts} (i).
    If $A$ is free then we are done by
  Sections 4 and 5 as in the proof of Theorem \ref{psdfthm}
  with the exceptions being cyclic or $BS(1,n)$.
If not then $A$ must be of cohomological dimension
  2 because $G$ is by Proposition \ref{1rfcts} (iii).

  We now apply \cite[Corollary 2.5]{gmsw} which states the following.
  Let $A$ be
  of type $FP_2$ with cohomological dimension 2. Let $\phi:A\rightarrow A$
  be a monomorphism and let $H$ denote $A*_\phi$. Then the cohomological
  dimension of $H$ is 3. But this is a contradiction so $A$ must be free.

  In the case where $G$ is a 1-relator group with torsion, we could try
  to mimic that result by dropping to an appropriate finite index subgroup
  of $G$ which is torsion free. However we may as well use the recent
  very powerful results in \cite{lrswlt}. There Theorem 5.1 states
  that if $G$ is a 1-relator group with torsion then there exists a
  torsion free finite index subgroup $G_0$ which is the fundamental
  group of a finite 2-complex $X_0$ with uniform negative immersions.
  But then Theorem 6.5 (iii) of that paper states that any finitely
  generated subgroup of $G_0$ is either cyclic or has a
  presentation of deficiency at least two. Thus 
if $H$ is any finitely generated subgroup of $G$
then it will have the finite index subgroup
$H\cap G_0$ of $H$ which is cyclic or is not boundedly generated.
In the second case $H$ will virtually have both of our
rank 1 properties. As for
the first case, it was shown in 
\cite[Theorem 5]{krsl} that the only virtually
soluble subgroups of a 1-relator group with torsion are
(infinite or finite) cyclic groups or $C_2*C_2$.
\end{proof}

This gives us:
\begin{cor} \label{1ls}
  If $H$ is any subgroup of a 1-relator group $G$ then $H$ is
  boundedly generated precisely when it is 
cyclic, $C_2*C_2$ or a soluble Baumslag-Solitar group
  $BS(1,n)$.
\end{cor}  
\begin{proof}
  If $H$ is not finitely generated then it cannot be boundedly generated.
  Otherwise apply Theorem \ref{rel1sb} and Proposition \ref{ntbg}.
\end{proof}
  
Having seen that all 1-relator groups apart from the ``small'' cases
have at least one of our two rank 1 phenomena, we might want to see
if this still holds on strengthening the first phenomenon from having
a finite index subgroup with
infinite dimensional second bounded cohomology to having a finite
index subgroup which is acylindrically
hyperbolic. A little thought reveals that Baumslag - Solitar groups
of the form $BS(m,n)$ where $m$ and $n$ are coprime will also be
excluded as we know they 
and their finite index subgroups neither are acylindrically hyperbolic
by Theorem \ref{gbs} (i),
nor have virtually sufficient abelianisation as was mentioned after
Definition \ref{vsa}. Also finite index subgroups of these
Baumslag -- Solitar
groups will provide other examples if they have a 1-relator presentation.
However surprisingly there are still more exceptions.

For this we can use:
\begin{prop} \label{1rel2gen} (\cite[Proposition 4.20]{motr}) 
Let $G$ be a group with two generators and one defining relator. Then at 
least one of the following holds:
\begin{itemize}
  \item[(i)] $G$ is acylindrically hyperbolic;
  \item[(ii)] $G$ contains an infinite cyclic $s$-normal subgroup.
  More precisely, either $G$ is infinite cyclic or it is an HNN-extension 
of the form $$G=\langle a,b, t\mid  a^t=b, w=1\rangle $$ of a 
2-generator 1-relator group
  $H= \langle a,b \mid w(a,b)\rangle $ with non-trivial center, so that 
$a^r=b^s$ in $H$ for some $r,s \in \Z \setminus\{0\}$. In the latter case $H$  
is free by cyclic and contains a finite index normal 
subgroup splitting as a direct product  
of a finitely generated free group with an infinite cyclic group.
  \item[(iii)] $G$ is isomorphic to an ascending HNN extension
of a finite rank free group.
\end{itemize}
Moreover, the possibilities (i) and (ii) are mutually exclusive.
\end{prop}

In (ii) \cite[Theorem C]{krp} tells us that $G$ is a generalised
Baumslag - Solitar group,
so that Theorem \ref{gbs} applies to say $G$ and all its
finite index subgroups
are not acylindrically hyperbolic.

Thus we need to determine when generalised
Baumslag - Solitar groups and their finite index subgroups
do not have virtually sufficient abelianisation.
They certainly do if the group is large. By \cite[Theorem 6.7]{lvrf},
if a generalised Baumslag - Solitar group $G$ is not large (nor cyclic)
then it may be defined by a reduced
graph that is a circle, with edges $e_i$ that
are consistently oriented and with non-zero integer labels $l_i$ and
$r_i$ at the left and the right of $e_i$ and such that $L=\Pi\,l_i$ and
$R=\Pi\, r_i$ are coprime.
Thus if our group is not a Baumslag - Solitar group
then our circle will have 
$n\gqqs 2$ edges with $l_i,r_i\neq 0,\pm 1$ as the graph is reduced,
whereupon
$G$ will have the deficiency 1 presentation
\begin{equation} \langle a_1,a_2,\ldots ,a_n,t\,|\,
  a_1^{l_1}=a_2^{r_1}, a_2^{l_2}=a_3^{r_2}, \ldots, a_n^{l_n}
  =ta_1^{r_n}t^{-1}\rangle
  \end{equation}
with $L=l_1\ldots l_n$ and $R=r_1\ldots r_n$ being coprime.

\begin{thm} \label{last} Suppose that
  $G$ is a non-cyclic
finitely generated group defined by a
  1-relator presentation. If $G$ is neither virtually
  acylindrically hyperbolic 
  nor has a finite index subgroup $H$ where for some prime $p$,
  $H_{\hat{p}}$ is not $p$-adic analytic then $G$ is a generalised
  Baumslag - Solitar group with a presentation of the form in (7.1).
  Conversely if $G$ has such a presentation then it is neither
virtually acylindrically hyperbolic nor has a finite index subgroup
  $H$ with $H_{\hat{p}}$ not $p$-adic analytic for some prime $p$.
  \end{thm}
\begin{proof}
First suppose $G$ has neither of the properties in the statement.
Then $G$ cannot be defined by a 1-relator presentation with three
or more generators by Corollary \ref{rk1pd2}. So we are in one
of the three cases in Proposition \ref{1rel2gen} but not
Case (i).
Then Case (iii) is dealt with using the results in Sections 4 and 5
  except when the free group has rank 0 or 1, giving us 
$\Z$ or $BS(1,n)$. Thus
  we are left in Case (ii) whereupon we now know that $G$ is
  a generalised Baumslag - Solitar group which is not large, because of
  its lack of virtually sufficient abelianisation, so $G$ has some presentation
  of the form in (7.1).

For the converse, any group given by such a presentation is a
generalised Baumslag -- Solitar group and so is not virtually
acylindrically hyperbolic by Theorem \ref{gbs} (i).
  So we now need to show that all groups defined by 
such a presentation
  fail to have virtually sufficient abelianisation. Note that
  \[a_i^{r_{i-1}\ldots r_1}=a_1^{l_{i-1}\ldots l_1}\mbox{ for any }
    2\lqqs i\lqqs n\] and
  \[ a_i^{l_i\ldots l_{n-1}l_n}=a_n^{r_i \ldots r_{n-1}l_n}=ta_1^{r_1\ldots r_n}t^{-1}
      \mbox{ for any } 1\lqqs i\lqqs n.\]
    As $r_{i-1}\ldots r_1$ and $l_i\ldots l_n$ are coprime, we have that
    $a_i\in\langle t,a_1\rangle$ so that $G$ is generated by $t$ and $a_1$.
    Moreover we have shown that the relation $ta_1^Rt^{-1}=a_1^L$ holds
    in $G$ so that $G$ is a quotient of $BS(R,L)$ for $R,L$ coprime. As the
    latter does not have virtually sufficient abelianisation, nor does the
    former.
\end{proof}

We now show that 2-generator 1-relator presentations can
occur for these groups.
\begin{cor}
  There exist generalised Baumslag -- Solitar groups $G$
  satisfying all of the following:\\
  $\bullet$ $G$ has a 2-generator 1-relator presentation\\
  $\bullet$   $G$  is neither virtually
  acylindrically hyperbolic 
  nor has a finite index subgroup $H$ where for some prime $p$,
  $H_{\hat{p}}$ is not $p$-adic analytic\\
  $\bullet$ $G$ is not a finite index subgroup of any $BS(p,q)$
  with $p$ and $q$ coprime.
\end{cor}
\begin{proof}
We say that a labelled graph defining a generalised Baumslag -- Solitar group
is strongly slide free if at each vertex, no integer from an incoming
edge divides that from any other incoming edge.
Then \cite{frdfrm} says that a generalised Baumslag -- Solitar group can
only be defined by at most one strongly slide free labelled graph (up to
the obvious changes of sign on the integers). In particular
this holds for $BS(p,q)$ with $p,q$ coprime and for groups with a
presentation as in (7.1). These groups also satisfy the middle of our three
properties in the statement of this corollary, by Theorem \ref{last}.

The paper \cite{ddkn} gives a complete description of the finite index
subgroups of $BS(p,q)$ for $p,q$ coprime. They are precisely when
$p=l_1=\ldots =l_n$ and $q=r_1=\ldots =r_n$ for the presentation in (7.1),
up to allowable changes of sign of the integers in the defining graph
and the usual reordering and changes of sign for $p$ and $q$.
Therefore provided we avoid having all $l_i$ constant and all $r_i$
constant (again up to allowable sign changes),
we are not isomorphic to a finite index subgroup of $BS(p,q)$
because all graphs involved will be strongly slide free, so different
defining graphs imply different groups.

So we have shown that the generalised Baumslag - Solitar groups $G$
satisfying the statement of this corollary are precisely (up to
allowable sign changes) those 
having a presentation of the form in (7.1) with either non-constant $l_i$
or non-constant $r_i$, and such that $G$ also has a 2-generator 1-relator
presentation. We finish by showing that such groups exist.

Let us take $n=2$ in the presentation (7.1) so that
we have
\[G:=\langle a,b,t\,|\,
  a^i=b^j, b^k
  =ta^lt^{-1}\rangle\]
where $ik$ is coprime to $jl$, none of $i,j,k,l$ are equal to $\pm 1$,
and $(k,l)\neq (i,j)$ or $(-i,-j)$.
Using a ``non-abelian'' Euclidean
algorithm on the coprime integers $j$ and $k$, we can perform Nielsen
moves on our presentation by reducing the powers
$b^j$ and $b^k$ until we are left with $b=u(a,t)$ and $b^m=v(a,t)$ for
some $m$, whereupon we can substitute $u$ into $b^m$, thus
eliminating $b$ and giving us a 2-generator 1-relator presentation
on $a$ and $t$.
\end{proof}

For instance taking
\[G:=\langle a,b,t\,|\,
  a^2=b^3, b^7
  =ta^5t^{-1}\rangle\]
gives us that $b=b^7b^{-6}=ta^5t^{-1}a^{-4}$ so that
$G=\langle t,a\,|\,a^2=(ta^5t^{-1}a^{-4})^3\rangle$.
On setting $u=tat^{-1}$, we see this is an HNN extension
of $H=\langle a,b,u\,|\,a^2=b^3, b^7=u^5\rangle$, namely
$\langle u,a\,|\,a^2=(u^5a^{-4})^3\rangle$ which is a 
2-generator 1-relator group with non-trivial center since $a^{14}=u^{15}$,
thus recovering the form of $G$ in Case (ii).

\end{document}